\documentclass[11pt]{amsart}

\usepackage[a4paper]{geometry}
\usepackage[plain]{fullpage}

\usepackage[utf8]{inputenc}
\usepackage{csquotes}

\usepackage{amsmath,amsthm,amssymb,amsfonts}
\usepackage{enumerate}

\usepackage{biblatex}
\usepackage[hidelinks]{hyperref}

\theoremstyle{plain}
\newtheorem{theorem}{Theorem}
\newtheorem{corollary}[theorem]{Corollary}
\newtheorem{proposition}[theorem]{Proposition}
\newtheorem{lemma}[theorem]{Lemma}
\newtheorem{conjecture}[theorem]{Conjecture}

\theoremstyle{definition}
\newtheorem{definition}[theorem]{Definition}

\theoremstyle{remark}
\newtheorem{remark}[theorem]{Remark}

\newcommand{\Z}{\mathbb{Z}}
\newcommand{\Q}{\mathbb{Q}}
\newcommand{\R}{\mathbb{R}}
\newcommand{\C}{\mathbb{C}}
\renewcommand{\O}{\mathcal{O}}
\newcommand{\p}{\mathfrak{p}}
\DeclareMathOperator{\ord}{ord}
\DeclareMathOperator{\lcm}{lcm}

\title{On the height of polynomials that split completely over a fixed number field}
\author{Thian Tromp}

\begin{document}

\begin{abstract}
    Let $K/\Q$ be a finite extension. We prove that the minimal height of polynomials of degree $n$ of which all roots are in $K^\times$ increases exponentially in $n$. We determine the implied constant exactly for totally real $K$ and $K$ equal to $\Q(\sqrt{-1})$ or $\Q(\sqrt{-3})$.
\end{abstract}

\maketitle

\section{Introduction}

It is a general empirical phenomenon that the irreducible factors of a polynomial in $\Z[x]$ with small coefficients cannot all have very low degree. Recently, Hajdu, Tijdeman and Varga \cite{hajdu} looked at polynomials in $\Z[x]$ that split completely over $\Q$ and showed the following.

\begin{theorem}[Hajdu-Tijdeman-Varga {\cite[Theorem 1.1]{hajdu}}]
    Let $f = a_nx^n + \ldots +a_1x + a_0 \in \Z[x]$ be a polynomial of degree $n$ of which all roots are in $\Q^\times$. Let $H = \max_{1 \leq i \leq n} |a_i|$. We then have
    \[
    n - o(n) \leq \frac{2}{\log 2} \log H.
    \]
    Moreover, $\frac{2}{\log 2}$ cannot be replaced by a smaller constant.
\end{theorem}

In this paper, we will generalise this theorem to number fields $K$ (finite extensions of $\Q$). We will compare the height (see Section \ref{preliminaries}) of a polynomial that splits completely over $K$ to its degree. More precisely, we will obtain general bounds on the constants defined below, as well as determine their exact value in infinitely many cases.

\begin{definition}\phantomsection\label{defCK}
    For a number field $K$, denote by $C_K$ the infimum of the set of constants $C$ such that for all polynomials of degree $n$ and height $H$ of which all roots are in $K^\times$, we have
    \[
        n - o(n) \leq C \log H
    \]
    (where the infimum of the empty set is defined to be $\infty$).
\end{definition}

The following is the main result of this paper.

\begin{theorem}
    \phantomsection\label{t1}
    Let $K$ be a number field of degree $d$ and containing $w_K$ roots of unity. Then the following statements hold.
    \begin{enumerate}[(i)]
        \item
        \phantomsection\label{t1(i)}
        We have
        \[
            \frac{w_K}{\log 2}\leq C_K \leq \frac{w_K}{\log 2} + \frac{d}{\log M_K},
        \]
        where $M_K$ is the minimal Mahler measure greater than $1$ of the polynomials $f^\alpha_\Z(x)^{[K:\Q(\alpha)]}$ for $\alpha \in K$, where $f^\alpha_\Z$ is the minimal polynomial of $\alpha$ over $\Z$ (see Section \ref{preliminaries}).
        \item
        \phantomsection\label{t1(ii)}
        If $K$ is totally real or equal to $\Q(\sqrt{-1})$ or $\Q(\sqrt{-3})$, we have
        \[
            C_K = \frac{w_K}{\log 2}.
        \]
        This equality does not hold for all number fields $K$.
        \end{enumerate}
\end{theorem}

We will discuss some of the rich literature on (bounds of) Mahler measures in Section \ref{preliminaries}.

Our techniques also work in a bit more generality, when $K$ is replaced by other subsets of $\overline{\Q}$. As an example, we give the following.

\begin{theorem}
    \phantomsection\label{t2}
    Let $k$ be a positive integer. There exists a constant $C$ such that for all polynomials $f \in \overline{\Q}[x]$ of degree $n$ and height $H$ of which all roots are nonzero and have degree at most $k$ over $\Q$, we have
    \[
    n - o(n) \leq C \log H.
    \]
\end{theorem}

This last theorem gives a precise statement which partly explains the phenomenon described at the start of this introduction. The challenge of finding the true value of $C_K$ for more number fields $K$ or other sets remains.

In Section \ref{preliminaries} we will discuss the relevant background knowledge and statements of results in our context. In Section \ref{mainresults} we will prove the theorems we have discussed in this introduction.

\section{Preliminaries}\label{preliminaries}
    In this section, we discuss the known definitions and results that we need. We will freely use some basic notions in algebraic number theory. Most of the material in this section is covered in more depth in \cite[Chapter 1]{bombieri}, although we deviate slightly from the notation there in some cases.

    Throughout this section, let $K$ be a number field of degree $d$ and with $w_K$ roots of unity. We consider the following normalisation of the absolute values on $K$. For each nonzero prime ideal $\p$ of the ring of integers $\O_K$ of $K$, we define the discrete absolute value $|\cdot|_\p$ by
    \[
        |x|_\p = N(\p)^{-\ord_\p(x)}
    \]
    for all $x \in K^\times$, where $N(\p)=[\O_K : \p]$ is the norm of $\p$ and $\ord_\p(x)$ is the exponent of $\p$ in the factorisation of $(x)$ into prime ideals of $\O_K$. For each embedding $\sigma\colon K \to \C$, we define the archimedean absolute value $|\cdot|_\sigma$ by
    \[
        |x|_\sigma = |\sigma(x)|_\C,
    \]
    where $|\cdot|_\C$ is our notation for the usual absolute value on $\C$. These absolute values obey the so-called product formula.

    \begin{proposition}[Product formula]
        For each $x \in K^\times$, we have
        \[
            \prod_{\p}|x|_\p \prod_{\sigma}|x|_\sigma = 1,
        \]
        where the products run over all nonzero prime ideals $\p$ in $\O_K$ and all embeddings $\sigma \colon K \to \C$ respectively.
    \end{proposition}

    For the rest of the paper, products over $\p$ or $\sigma$ are implied to be over all nonzero prime ideals $\p$ of $\O_K$ and embeddings $\sigma \colon K \to \C$, unless otherwise stated. Notice that for an embedding $\sigma\colon K \to \C$ with $\sigma(K) \not\subset \R$, the conjugate embedding $\overline{\sigma}$ gives rise to the same absolute value, which thus appears twice in such a product. Since we will explicitly analyse the embeddings, we keep using this (somewhat nonstandard) notation.

    The absolute values are used to define the height of a polynomial with coefficients in $K$ as follows.

    \begin{definition}
        Let $f = a_nx^n + \ldots + a_1x+a_0 \in K[x]$ be a polynomial and let $|\cdot|_v$ be an absolute value on $K$. We define the Gauss norm $|f|_v$ of $f$ as
        \[
            |f|_v = \max_{1 \leq i \leq n} |a_i|_v.
        \]
        We define the height $H(f)$ of $f \in K[x]$ through
        \[
            H(f)^d = \prod_\p |f|_\p \prod_\sigma |f|_\sigma.
        \]
    \end{definition}

    The height as defined above does not actually depend on the underlying number field $K$ (as long as it contains all coefficients, of course).

    Gauss's lemma says that the discrete Gauss norms are multiplicative.

    \begin{lemma}[Gauss's lemma, {\cite[Lemma 1.6.3]{bombieri}}]
        For each prime $\p$ of $K$ we have
        \[
            |fg|_\p = |f|_\p|g|_\p.
        \]
        In particular, if $f = \prod_{i=1}^n(x-\alpha_i)$ we have
        \[
            |f|_\p = \prod_{i=1}^n \max(1, |\alpha_i|_\p).
        \]
    \end{lemma}

    This proposition shows that discrete Gauss norms of polynomials are easy to calculate from their roots (which is what we are concerned with). For the archimedean Gauss norms, this is not as easy and to handle these cases, we consider the so-called Mahler measure of a complex polynomial.

    \begin{definition}
        The Mahler measure $M(f)$ of a polynomial $f = a(x-\alpha_1)\cdots (x-\alpha_n) \in \C[x]$ is defined as
        \[
        M(f) = a\prod_{i=1}^n \max(1, |\alpha_i|_\C).
        \]
    \end{definition}

    It turns out that the Mahler measure is related to a complex line integral over the unit circle; more precisely, we have $M(f) = \exp(\int_0^{1}\log(f(e^{2\pi i t}))dt)$. This identification allows us to give estimates on the usual complex Gauss norm of a polynomial. In particular, we have the following.

    \begin{proposition}[{\cite[Lemma 1.6.7]{bombieri}}]
        \phantomsection\label{complexmahler}
        For any polynomial $f \in \C[x]$ we have
        \[
            |f|_\C \geq M(f)(n+1)^{-1/2}.
        \]
    \end{proposition}

    If $K$ is a number field and $f = a_nx^n + \ldots + a_1x + a_0 \in K[x]$, we can apply the above proposition to the polynomial $f_\sigma = \sigma(a_n) x^n + \ldots + \sigma(a_1)x + \sigma(a_0) \in \C[x]$ for each embedding $\sigma \colon K \to \C$, take the product of all inequalities and the Gauss norms $|f|_\p$ for all primes $\p$ of $K$ to obtain the following.

    \begin{proposition}
        \phantomsection\label{alphabound1}
        Let $f = a(x-\alpha_1)\cdots(x-\alpha_n) \in K[x]$ be a polynomial that splits completely over $K$. We then have
        \[
            H(f)^d \geq (n+1)^{-d/2}\prod_{i=1}^n \prod_\p \max(1,|\alpha_i|_\p) \prod_\sigma \max(1, |\alpha_i|_\sigma).
        \]
    \end{proposition}

    The above proposition motivates us to study the quantity
    \[
        \prod_\p \max(1, |\alpha|_\p) \prod_\sigma \max(1, |\alpha|_\sigma)
    \]
    for elements $\alpha \in K$. If we let $f^\alpha_\Z$ be the minimal polynomial of $\alpha$ over $\Z$, this is equal to the Mahler measure of $f^\alpha_\Z(x)^{[K:\Q(\alpha)]}$ (see \cite[Proposition 1.6.6]{bombieri}), sometimes called the characteristic polynomial of $\alpha$ for the extension $K/\Q$, and we will therefore denote it by $M_K(\alpha)$. It turns out that the set
    \[
    \{M_K(\alpha) \mid \alpha \in K, M_K(\alpha) \leq \beta\}
    \]
    is finite for all $\beta \in \R$ (as bounds on the archimedean absolute values of $\alpha$ give bounds on the coefficients of the characteristic polynomial of $\alpha$ in $K/\Q$, which are integers) and as a consequence $M_K(\alpha) = 1$ if and only if $\alpha$ is zero or a root of unity (as we have then have $M_K(\alpha^m)=1$ for all positive integers $m$).
    
    Denote by $M_K$ the minimum of the set $\{M_K(\alpha)\mid \alpha \in K, M_K(\alpha)>1\}$. The following is now immediate from Proposition \ref{alphabound1}

    \begin{proposition}
        \phantomsection\label{bound1}
        Let $f\in K[x]$ be a polynomial of degree $n$ of which all roots are in $K^\times$ and of which $r$ are roots of unity. Then, we have
        \[
            \log H(f) \geq \frac{n-r}{d}\log M_K - \frac{1}{2}\log(n+1)
        \]
    \end{proposition}

    There has been quite a bit of research on the values of $M_K(\alpha)$ and thus of $M_K$ (although typically in the context where we always let $K = \Q(\alpha)$ and simply vary $\alpha$ over $\overline{\Q}$, i.e. considering the Mahler measure of the minimal polynomial of an algebraic number; see Smyth \cite{smyth2} for a survey). There is evidence for the following conjecture, which was first posed as a question by Lehmer \cite{lehmer}.

    \begin{conjecture}[Lehmer's problem]
        \phantomsection\label{lehmer}
        For any $\alpha \in \overline{\Q}$ we have
        \[
            M_{\Q(\alpha)}(\alpha) \geq M(x^{10}+x^9-x^7-x^6-x^5-x^4-x^3+x+1) \approx 1.176.
        \]
    \end{conjecture}

    Let us state two of the most prominent partial results in the direction of this conjecture explicitly, since they can be used to make our bounds more explicit.

    \begin{proposition}[Smyth \cite{smyth}]
        For any $\alpha \in \overline{\Q}$ whose minimal polynomial is not reciprocal, we have
        \[
            M_{\Q(\alpha)}(\alpha) \geq M(x^3-x-1) \approx 1.325.
        \]
    \end{proposition}

    Notice that if $\alpha$ has odd degree over $\Q$ (so for $\alpha$ an element of a number field of odd degree in particular), its minimal polynomial can never be reciprocal, so the above bound holds in that case. More generally, we have the following bound.

    \begin{proposition}[Voutier \cite{voutier}]
        \phantomsection\label{voutier}
        For any $\alpha \in \overline{\Q}$ with $\deg \alpha \geq 2$, we have
        \[
            M_{\Q(\alpha)}(\alpha) \geq 1+\frac{1}{4}\left(\frac{\log\log \deg \alpha}{\log \deg \alpha}\right)^3.
        \]
    \end{proposition}

\section{Main results}\label{mainresults}

As in the previous section, let $K$ be a number field of degree $d$ with $w_K$ roots of unity. The previous section has provided us with lower bounds on the height of a polynomial of which all roots are in $K^\times$ and of which not too many roots are roots of unity (Proposition \ref{bound1}), based on some general theoretical machinery. The aim of this section is to give bounds on the height of polynomials of which many of the roots are roots of unity. Together, the bounds will give us a general upper bound for $C_K$. Moreover, we will show that for some $K$ the computations in this section suffice to find $C_K$ exactly, while for other $K$ they fundamentally cannot do so. We will also sketch a proof of Theorem \ref{t2}.

Let us start by giving a new bound on the height of a polynomial that splits over $K$ in terms of its roots.

\begin{lemma}
    \phantomsection\label{alphabound2}
    Let $f = a(x-\alpha_1)\cdots(x-\alpha_n)\in K[x]$ be a polynomial that splits completely over $K$. We then have
    \[
    H(f)^{dw_K} \geq (n+1)^{-dw_K} \prod_{i=1}^n \prod_\p \max(1,|\alpha_i|^{w_K}_\p)\prod_\sigma |1+\alpha_i^{w_K}|_\sigma.
    \]
\end{lemma}
\begin{proof}
    Using the product formula, we can assume without loss of generality that $f$ is monic, so $f = \prod_{i=1}^n(x-\alpha_i)$. Let $\zeta_{2w_K}$ be a primitive $2w_K$-th root of unity and consider the quadratic extension $K(\zeta_{2w_K})$ of $K$. Notice that
    \[
        \prod_{i=1}^n \prod_\sigma |\alpha_i^{w_K} + 1|_\sigma = \prod_{i=1}^n \prod_\sigma \prod_{j=1}^{w_K} |\alpha_i - \zeta_{2w_K}^{2j-1}|_\sigma = \prod_\sigma \prod_{j=1}^{w_K} |f(\zeta_{2w_K}^{2j-1})|_\sigma
    \]
    where the products range over all embeddings $\sigma \colon K(\zeta_{2w_K})\to \C$. The triangle inequality immediately shows that the absolute value of a polynomial evaluated at a root of unity is at most $n+1$ times the corresponding Gauss norm, so the above becomes
    \[
    \prod_{i=1}^n \prod_\sigma |\alpha_i^{w_K} + 1|_\sigma \leq (n+1)^{2dw_K}\prod_\sigma |f|_\sigma^{w_K},
    \]
    where the products still range over all embeddings $\sigma \colon K(\zeta_{2w_K})\to \C$. Since every embedding $\sigma\colon K \to \C$ extends into precisely two distinct embeddings $\sigma \colon K(\zeta_{2w_K}) \to \C$ and the numbers $\alpha_i^{w_K}+1$ for $1 \leq i \leq n$, as well as the coefficients of $f$, are actually in $K$, taking the square root of the previous inequality yields
    \[
        \prod_{i=1}^n \prod_\sigma |\alpha_i^{w_K} + 1|_\sigma \leq (n+1)^{dw_K}\prod_\sigma |f|_\sigma^{w_K},
    \]
    where the products range over all embeddings $\sigma\colon K \to \C$ again. Multiplying this inequality by $\prod_{\p}|f|_\p^{w_K}$ (and using Gauss's lemma) finishes the proof.
\end{proof}

The lemma above is strikingly similar to Proposition \ref{alphabound1}. One key difference is that this bound shows that roots of unity as roots of $f$ do contribute to the height. We obtain the following.

\begin{corollary}
    \phantomsection\label{bound2}
    Let $f\in K[x]$ be a polynomial of degree $n$ of which all roots are in $K^\times$ and of which $r$ of the roots are roots of unity. We then have
    \[
        \log H(f) \geq \frac{r}{w_K} \log 2 - \log(n+1).
    \]
\end{corollary}
\begin{proof}
    Consider the quantity
    \[
        \prod_\p\max(1,|\alpha|^{w_K}_\p)\prod_\sigma |\alpha^{w_K}+1|_\sigma
    \]
    for an $\alpha \in K^\times$. By the non-archimedean triangle inequality for each $\p$, we find this is at least
    \[
        \prod_\p|\alpha^{w_K}+1|_\p\prod_\sigma |\alpha^{w_K}+1|_\sigma
    \]
    for any $\alpha \in K^\times$, which is $1$ by the product formula. On the other hand, if $\alpha$ is a root of unity in $K$, so a $w_K$-th root of unity, the expression simply evaluates to $2^d$. By Lemma \ref{alphabound2}, it follows that
    \[
        H(f)^{dw_K} \geq (n+1)^{dw_K}\cdot 2^{dnr},
    \]
    from which the proposition follows.
\end{proof}

We are also interested in how close the height of a polynomial can actually be to the lower bounds we derive, i.e. we want to find lower bounds for $C_K$. The following elementary lemma helps us do this by considering simple families of polynomials.

\begin{lemma}
    \phantomsection\label{lowerbound}
    For a polynomial $f = a_nx^n + \ldots + a_1x + a_0 \in \Z[x]$ with $\gcd(a_1, \ldots, a_n)=1$, we have
    \[
        H(f^m) \leq (|a_n| + \ldots + |a_0|)^m
    \]
    for all integers $m \geq 0$. In particular, if $f$ splits completely over $K$, we have
    \[
        C_K \geq \frac{n}{\log(|a_n| + \ldots + |a_0|)}
    \]
\end{lemma}
\begin{proof}
    By the coprimality condition on the coefficients of $f$, $|f|_\p = 1$ for all $\p$, so by Gauss's lemma all discrete Gauss norms are equal to $1$. All archimedean Gauss norms are equal to $|f^m|$, so we have $H(f^m) = |f^m|$. Because the sum of the absolute values of the coefficients is larger than their maximum, we get
    \[
        H(f^m) = |f^m| \leq \sum_{\ell=0}^{mn}\left(\left|\sum_{\ell_1 +\ldots + \ell_m = \ell}\left(\prod_{j=1}^m a_{\ell_j}\right)\right| x^\ell\right) \leq \sum_{\ell=0}^{mn}\left(\sum_{\ell_1 +\ldots + \ell_m = \ell}\left(\prod_{j=1}^m |a_{\ell_j}|\right)x^\ell\right)
    \]
    by the triangle inequality. This is the sum of the coefficients of $(|a_n|x^n + \ldots +|a_1|x +|a_0|)^m$ so equal to that polynomial evaluated at $1$, from which the result follows.
\end{proof}

Let us move on to the proof of Theorem \ref{t1}, starting with the first part.

\begin{proof}[Proof of Theorem {\hyperref[t1(i)]{\ref*{t1}(\ref*{t1(i)})}}.]
    The bound $\frac{w_K}{\log 2} \leq C_K$ follows from Lemma \ref{lowerbound} with $f = x^{w_K}-1 = \prod_{i=1}^{w_K}(x-\zeta_{w_K}^i)$. The upper bound for $C_K$ follows from Proposition \ref{bound1} and Lemma \ref{alphabound2}: if we multiply the bound in the former by $\frac{w_K}{\log 2}$ and the one in the latter by $\frac{d}{\log M_K}$ and add the inequalities, we find that for any $f \in K[x]$ of degree $n$ with all roots in $K^\times$ we have
    \[
        \left(\frac{w_K}{\log 2}+\frac{d}{\log M_K}\right) \log H(f) \geq n - \left(\frac{w_K}{2\log 2} + \frac{d}{\log M_K}\right)\log(n+1),
    \]
    from which the result follows.
\end{proof}

\begin{remark}
    It is natural to wonder what the polynomials (of which all roots are in $K^\times$) of minimal height for their degree are. Our arguments are mostly able to bypass having to answer this question, but we do not really build any intuition for what these polynomials look like as a result. While the interaction with the underlying number fields is hard to understand completely, we do have a reasonable understanding of where the roots of a complex polynomial need to be for its coefficients to be small. Firstly, we have already seen (Proposition \ref{complexmahler}) that the roots should not have large absolute value in this case. Secondly, it is known that the arguments of the roots should be spread out along the interval $[0,2\pi)$, since many roots of which the arguments are bunched together force large coefficients (see e.g. the work of Erd\H{o}s and Turán \cite{erdos}). This provides heuristics for constructing polynomials of small height: one needs to ensure that the roots are small and spread out in all complex embeddings. Moreover, this gives a more intuitive reason for why we should have $C_K < \infty$: roots of a polynomial that are not roots of unity contribute to the height because of their size, while those that are roots of unity contribute because they are forced to bunch up if there are enough, since any $K$ only contains finitely many roots of unity. Such ideas actually provide an alternative way to give upper bounds on $C_K$. It turns out the bounds obtained like this are worse than those we present here. More details may be found in the author's Bachelor's thesis \cite{tromp}.
\end{remark}

For some number fields $K$, the bound in Lemma \ref{alphabound2} is enough to give an upper bound on $C_K$, without needing a result like Proposition \ref{alphabound1}. In fact, the exact value of $C_K$ can be derived from it in some cases. We proceed to the proof of the second part of Theorem \ref{t1}.

\begin{proof}[Proof of Theorem {\hyperref[t1(i)]{\ref*{t1}(\ref*{t1(ii)})}}.]
    We have to show that $C_K \leq \frac{w_K}{\log_2}$ for $K$ totally real or equal to $\Q(\sqrt{-1})$ or $\Q(\sqrt{-3})$. It is enough to show that the quantity
    \[
        \prod_{i=1}^n \prod_\p \max(1,|\alpha_i|^{w_K}_\p)\prod_\sigma |1+\alpha_i^{w_K}|_\sigma
    \]
    on the right hand side of the bound in Lemma \ref{alphabound2} is at least $2^{dn}$, which we will do by showing that
    \[
        \prod_\p \max(1,|\alpha|^{w_K}_\p)\prod_\sigma |1+\alpha^{w_K}|_\sigma \geq 2^d
    \]
    for all $\alpha \in K^\times$.

    \emph{Case 1: totally real $K$.} Notice that $w_K=2$. For each $\alpha \in K^\times$ and each embedding $\sigma \to \R$ we have 
    \[
    |1+\alpha^{2}|_\sigma = |1+\sigma(\alpha)^2|_\R \geq 2|\sigma(\alpha)|_\R = 2|\alpha|_\sigma
    \]
    by the inequality of the arithmetic and geometric means. Since $\max(1,|\alpha|^{2}_\p) \geq |\alpha|_\p$ for all $\p$, it follows that
    \[
        \prod_\p \max(1,|\alpha|^{2}_\p)\prod_\sigma |1+\alpha^{2}|_\sigma \geq \prod_\p |\alpha|_\p\prod_\sigma 2 |\alpha|_\sigma = 2^d
    \]
    by the product formula.

    \emph{Case 2: $K = \Q(\sqrt{-1})$.} Notice that $w_K=4$. Let $\alpha \in K^\times$. Since $\O_K$ is a UFD, we can write $\alpha = \beta/\gamma$ for $\beta,\gamma \in \O_K$ coprime and both nonzero. Then simply considering the prime factorisation, we have $\max(1,|\alpha|_\p) = |1/\gamma|_\p$ for any $\p$, so that
    \[
        \prod_\p \max(1,|\alpha|^{4}_\p)\prod_\sigma |1+\alpha^{4}|_\sigma = \prod_\p |1/\gamma|_\p^4 \prod_\sigma |1/\gamma|^4_\sigma|\beta^4+\gamma^4|_\sigma = \prod_\sigma |\beta^4+\gamma^4|_\sigma
    \]
    by the product formula, so we are left to show that $\prod_\sigma |\beta^4+\gamma^4|_\sigma \geq 2^d = 4$. For a general element $a + b\zeta_4$ in $\O_K$, the binomial theorem shows
    \[
        (a+b\zeta_4)^4 \equiv (a^2+b^2)^2 \mod 4,
    \]
    so that a fourth power in $\O_K$ is congruent to $0$ or $1$ modulo $4$, so that $\beta^4+\gamma^4$ is congruent to $0$, $1$ or $2$ modulo $4$. From considering $\O_K$ as a lattice in the complex plane, it is clear that the only elements of $\O_K$ congruent to $0$, $1$ or $2$ modulo $4$ with norm less than $4$ (i.e. within the circle of radius $2$ centered at the origin) are $0$ and $1$. Suppose for a contradiction that $\prod_\sigma |\beta^4+\gamma^4|_\sigma < 4$. If $\beta^4 +\gamma^4 = 0$, we get $(\beta/\gamma)^4=-1$, so $K$ contains a primitive eighth root of unity, which is a contradiction. If instead
    \[
        (\beta^2-\zeta_4\gamma^2)(\beta^2+\zeta_4\gamma^2) = \beta^4 + \gamma^4 = 1,
    \]
    we get
    \[
        (\beta^2-\zeta_4\gamma^2,\beta^2+\zeta_4 \gamma^2) \in \{(1,1),(-1,-1),(\zeta_4,-\zeta_4),(-\zeta_4,\zeta_4)\}.
    \]
    In the first two cases, subtracting the two factors gives $\gamma=0$, a contradiction. In the last two cases, adding the two factors gives $\beta=0$, also a contradiction.

    \emph{Case 3: $K = \Q(\sqrt{-3})$.} Notice that $w_K=6$. Let $\alpha \in K^\times$. Analogously to the previous case, $\O_K$ is a UFD and we can write $\alpha = \beta/\gamma$ for $\beta,\gamma \in \O_K$ coprime and both nonzero and find
    \[
        \prod_\p \max(1,|\alpha|^{6}_\p)\prod_\sigma |1+\alpha^{6}|_\sigma = \prod_\p |1/\gamma|_\p^6 \prod_\sigma |1/\gamma|^6_\sigma|\beta^6+\gamma^6|_\sigma = \prod_\sigma |\beta^6+\gamma^6|_\sigma
    \]
    by the product formula, so we are left to show that $\prod_\sigma |\beta^6+\gamma^6|_\sigma \geq 2^d = 4$. For a general element $a + b\zeta_3$ in $\O_K$, the binomial theorem shows
    \[
        (a+b\zeta_3)^6 \equiv (a^3+b^3)^2 \mod 3,
    \]
    so that a sixth power in $\O_K$ is congruent to $0$ or $1$ modulo $3$, so that $\beta^6 + \gamma^6$ is congruent to $0$, $1$ or $2$ modulo $3$, so that $\beta^6+\gamma^6 \equiv 0,1,2 \mod 3$. From considering $\O_K$ as a lattice in the complex plane, it is clear that the only elements of $\O_K$ congruent to $0$, $1$ or $2$ modulo $3$ with norm less than $4$ (i.e. within the circle of radius $2$ centered at the origin) are $-1$, $0$ and $1$. Suppose for a contradiction that $\prod_\sigma |\beta^6+\gamma^6|_\sigma < 4$. If $\beta^6 +\gamma^6 = 0$, we get $(\beta/\gamma)^6=-1$, so $K$ contains a root of unity of order not dividing $6$, which is a contradiction. Otherwise we have
    \[
        (\beta^2+\gamma^2)(\beta^2+\zeta_3\gamma^2)(\beta^2+\zeta_3^2\gamma^2) = \beta^6+\gamma^6 = \pm1.
    \]
    We conclude each factor on the left is equal to $\pm 1, \pm \zeta_3, \pm \zeta_3^2$, since those are the units in $\O_K$. There must be two factors that have the same sign in our notation, so that they differ by a power of $\zeta_3$. In other words, we find $a,b,c \in \{0,1,2\}$ (with $a \neq c$) such that
    \[
    \beta^2+\zeta_3^a\gamma^2 = \zeta_3^b(\beta^2+\zeta^c\gamma^2) = \zeta_3^b\beta^2+\zeta_3^{b+c}\gamma^2,
    \]
    so
    \[
    (1-\zeta_3^b)\beta^2 = (\zeta_3^{b+c}-\zeta_3^a)\gamma^2.
    \]
    From here, it is not hard to check that either precisely one of the sides of this equation is $0$ (not both, since $a \neq c$), so that one of $\beta$ and $\gamma$ is $0$, which is a contradiction, or cubing the equation (possibly after dividing by $1-\zeta_3$) yields $\beta^6 = \pm \gamma^6$, which contradicts the fact their sum is $\pm 1$. This settles this case.

    Finally, we want to show that we do not have $C_K = \frac{w_K}{\log 2}$ for every number field $K$. This follows from Lemma \ref{lowerbound} for many choices of the polynomial $f$. In particular, we can take
    \[
        x^8 + 2 x^6 - 3 x^4 - 4 x^2 + 4 = \left((x-1)(x+1)(x-\sqrt{-2})(x+\sqrt{-2})\right)^2,
    \]
    which splits over $K=\Q(\sqrt{-2})$ and shows $C_K \geq \frac{8}{\log 14} > \frac{2}{\log 2} = \frac{w_K}{\log 2}$. Taking higher powers of $(x-1)(x+1)(x-\sqrt{-2})(x+\sqrt{-2})$ as $f$ yields better lower bounds, but of course requires more computation.
\end{proof}

\begin{remark}
    There could be more number fields $K$ for which the above proof strategy of analysing the right hand side of the bound in Lemma \ref{alphabound2} to find a good bound on $C_K$ could work, although it will likely be considerably harder than in the above cases. We would like to point out that the strategy does not work for the other quadratic number fields, i.e. $K=\Q(\sqrt{-d})$ for $d=2$ and squarefree $d \geq 5$. In those cases, we have $w_K=2$ and we thus want to analyse elements of the form $1+\alpha^2$ for $\alpha \in K^\times$. The theory of Pell's equation gives us nonzero $b, c \in \Z$ such that $b^2-dc^2 = 1$. If we let $\beta=b$ and $\gamma = c\sqrt{-d}$ and $\alpha = \beta/\gamma$ we see
    \[
        \prod_\p \max(1,|\alpha|^{2}_\p)\prod_\sigma |1+\alpha^{2}|_\sigma \leq \prod_\p |1/\gamma|_\p^2 \prod_\sigma |1/\gamma|^2_\sigma|\beta^2+\gamma^2|_\sigma = \prod_\sigma |\beta^2+\gamma^2|_\sigma = \prod_\sigma |b^2-dc^2|_\sigma = 1,
    \]
    so that there is no hope of giving any general bounds on $C_K$ in this way in these cases. We do not immediately see an alternative proof strategy to resolve these problems in a simple way.
\end{remark}

While our calculations are particularly nice when working over a number field $K$, the proofs do not really require this, and we can replace $K$ by a more general subset of $\overline{\Q}$ and still achieve similar bounds. As an example, we sketch the proof of Theorem \ref{t2}.

\begin{proof}[Sketch of proof of Theorem \ref{t2}]
    We first pick a number $w$ such that $\zeta^w = 1$ for all roots of unity of degree at most $k$, e.g. $\lcm\{\ell\mid \phi(\ell) \leq k\}$ suffices. Mimicking the proofs of Lemma \ref{alphabound2} and Corollary \ref{bound2} with $w$ instead of $w_K$ (but where the number field over which the calculations are performed depends on $f$) shows that
    \[
        \log H(f) \geq \frac{r}{w}\log 2 - \log(n+1),
    \]
    where $r$ is the number of roots of $f$ that are roots of unity. Using Proposition \ref{voutier}, we can find find a constant $M$ such that
    \[
        \log H(f) \geq \frac{n-r}{k} \log M - \frac{1}{2}\log(n+1),
    \]
    a bound similar to Corollary \ref{bound2}. Combining these two bounds as in the proof of Theorem {\hyperref[t1(i)]{\ref*{t1}(\ref*{t1(i)})}} gives us the desired constant.
\end{proof}

\section*{Acknowledgements}

This paper is essentially a translated summary of the results in the author's Bachelor's thesis, which was supervised by Gunther Cornelissen at Utrecht University. The author thanks him for his guidance.

\printbibliography

\end{document}